\documentclass[12pt,reqno]{amsart}
\usepackage{amssymb,amsfonts,amsthm,amsmath}

\usepackage{latexsym,multicol,graphicx}

\usepackage{graphicx, verbatim}

\usepackage{xcolor}
\usepackage{verbatim}

\usepackage{pict2e}

\usepackage{fancyhdr}
\usepackage{enumerate}
\usepackage{xcolor}
\usepackage{relsize}

\newcommand\spanmat{{\bigvee}}

\newcommand\dist{{\text{\rm\,dist}}}

\newtheorem{theorem}{Theorem}
\newtheorem{lemma}[theorem]{Lemma}
\newtheorem{corollary}[theorem]{Corollary}
\newtheorem{proposition}[theorem]{Proposition}

\newtheorem*{theoremletterA}{Theorem A}
\newtheorem*{lemmaletterB}{Lemma B}
\newtheorem*{theoremletterC}{Theorem C}
\newtheorem*{theoremletterD}{Theorem D}

\newtheorem*{propositionletter?}{Proposition E}
\newtheorem*{lemmaletterF}{Lemma F}
\newtheorem*{lemmaletterG}{Lemma G}
\newtheorem*{lemmaletterH}{Lemma H}

\newtheorem*{theoremUK}{Theorem on $U+\mathfrak{S}_\infty$ families}
\newtheorem*{theoremAOS}{Theorem on AOS}

\newtheorem*{AAK}{Adamyan-Arov-Krein Theorem}

\newtheorem*{thmHa}{Hartman's theorem}

\newtheorem*{ISP3}{Inverse Schur-Nevanlinna process III}

\newtheorem{rem}{Remark}

\newcommand{\D}{{\mathbb D}}
\newcommand{\T}{{\mathbb T}}

\newcommand{\te}{\theta }

\begin{document}
\title[Schur--Nevanlinna parameters]{Schur--Nevanlinna parameters, Riesz bases, and compact Hankel operators on the model space}
\author{Inna Boricheva}



\address{\!\!\!\!\!\!22 avenue Elleon, 
Marseille, France
\newline {\tt innaboritchev@gmail.com}
}

\begin{abstract}
We study Riesz bases/Riesz sequences
of reproducing kernels in the model space $K_\te$ in connection with the corresponding  Schur--Nevanlinna parameters and functions. 
In particular, we construct inner functions with given Schur--Nevanlinna parameters at a given sequence $\Lambda$ such that the corresponding systems of projections of reproducing kernels in the model space
are complete/non complete. Furthermore, we give a compactness criterion for Hankel operators with symbol $\theta \overline B$, where $\theta$ is an inner function and $B$ is an interpolating Blaschke product 
and use this criterion to describe Riesz bases
$\mathcal K_{\Lambda, \theta}$, with $\lim_{\lambda \in \Lambda, |\lambda| \to 1} \theta(\lambda) =0$.
\end{abstract}

\maketitle
\section{Introduction and main results}
Let $\theta$ be an inner function in the sense of Beurling, $\Lambda = \{\lambda_k\}_{k \ge 0}$
be a sequence of distinct points in the open unit disc  $\D$ satisfying the Blaschke condition, 
\begin{equation}
\sum_{\lambda\in\Lambda}(1-|\lambda|)<\infty, 
\tag{B}
\end{equation} 
(from now on we write $\Lambda \in (B)$ if a set $\Lambda$ satisfies the Blaschke condition $(B)$),
and let $B_\Lambda$ be the corresponding Blaschke product. We are interested in geometric properties ((uniform) minimality, being a Riesz sequen\-ce/basis, being an asymptotically orthonormal basis) of the reproducing kernel family 
$$
\mathcal K_{\Lambda, \theta}=\Bigl\{k_\theta(\lambda, z) =P_\theta k(\lambda, z) = \frac {1-\overline {\theta(\lambda)} \theta(z)} {1-\overline {\lambda} z}\Bigr\}_{\lambda \in \Lambda}
$$ 
in the model space $K_\theta=H^2 \ominus \theta H^2$.
Here $P_\theta$ 
is the operator of orthogonal projection 
onto $K_\theta$, 
and $k(\lambda, z) =  \frac {1} {1-\overline {\lambda} z}$ is the reproducing kernel in the Hardy space $H^2$. B.Pavlov noticed (see  \cite[Vol.2, Part D, Section 4.4]{11}, \cite{2}, \cite{7}) that if  $\sup_{\lambda \in \Lambda} |\theta(\lambda)| <1$ and $P_\theta\vert K_{B_\Lambda}$ is an isomorphism onto its image, then the family $\mathcal K_{\Lambda, \theta}$ inherits some geometric properties of the family $\mathcal K_{\Lambda}=\{k(\lambda, \cdot) : \lambda \in \Lambda\} $.
Furthermore,  
$P_\theta\vert K_{B_\Lambda}$ is an isomorphism onto its image
if and only if 
\begin{equation}
\dist_{L^\infty}(\theta, B_{\Lambda}H^\infty)<1.
\label{distless1}
\end{equation} 

Notice that under condition \eqref{distless1}, we have $\sup_{\lambda \in \Lambda} |\theta(\lambda)| <1$.


In \cite{4} and \cite{5} we studied relationships between condition \eqref{distless1} and the so called Schur--Nevanlinna coefficients and functions. 
Given $\mu \in \D$, denote by  $\tau_\mu$ the corresponding M\"obius transformation, $\tau_\mu(z)=\frac{z-\mu}{1-\overline \mu z}$. Fix a function $\theta$ in the closed unit ball $\mathcal B$ of $H^\infty$ and a sequence $\Lambda=\{\lambda_n\}_{n\ge 0}$ of points in the unit disc.
 We define the
Schur--Nevanlinna functions $\{\theta_n\}_{n \ge 0}$ and the Schur--Nevanlinna coefficients $\{\gamma_n\}_{n \ge 0}$ as follows: 
$\theta_0 =\theta$, $\gamma_0 =\theta_0(\lambda _0)$ and 
\begin{equation}
\theta_n =  \frac{\tau_{\gamma_{n-1}}(\theta_{n-1})}{\tau_{\lambda_{n-1}}},\qquad 
\gamma_{n}  =\theta_{n}(\lambda _{n}),\qquad n\ge 1.
\label{Schur}                             
\end{equation}
The Schwarz lemma gives immediately that $\theta_n\in\mathcal B$, $n\ge 1$.
 
In \cite{4} we established that if 
\begin{equation}
\sum_{n \ge 0} |\gamma_n| < \infty,
\label{sumConverges}
\end{equation} 
then condition \eqref{distless1} is fulfilled.

Given $\mu \in \D \setminus \Lambda$, we set $l_\mu = \frac {\sqrt{1-|\mu|^2}} {1-\overline \mu z} B_\Lambda \in K_{B_{\Lambda \bigcup \{\mu\}}}$. In  \cite[Corolary 3.5]{5}  we showed that 
$$
\dist_{H^2}^2(P_\theta l_{\mu}, \spanmat\mathcal  K_{\Lambda , \theta}) =  
(1-|\theta(\mu)|^2) \prod_{k=1}^{{\infty}} \frac {1-|\theta_k(\mu)|^2}{1-|\theta_k(\lambda_{n})|^2|b_{\lambda_{k-1}}(\mu)|^2},
$$
where $\spanmat X$ denotes the closed linear hull of a family $X$.

Using \eqref{usefulidentity} below, we get
\begin{equation}
\dist_{H^2}^2(P_\theta l_{\mu}, \spanmat\mathcal  K_{\Lambda , \theta}) =  
\lim_{n\to\infty} (1-|\theta_n(\mu)|^2)  \prod_{k=0}^{n-1} \frac {|1-\overline \gamma_k \theta_k(\mu)|^2}{1- | \gamma_k|^2}.
\label{distlmu}
\end{equation}

Since $l_\mu$ are unit vectors orthogonal to $K_{B_\Lambda}$, we have $\dist( l_{\mu}, K_{B_\Lambda})=1$.  If $P_\theta\vert K_{B_{\Lambda \bigcup \{\mu\}}}$ is an isomorphism 
 into $K_\theta$, then the distance between the corresponding images
 is separated from $0$, that is
$$
\dist_{H^2} (P_\theta l_{\mu}, \spanmat\mathcal  K_{\Lambda , \theta}) >0.
$$

Under condition \eqref{sumConverges}  the latter means that $ \lim_{n \to \infty} |\theta_n(\mu)| <1$.

In Section 3 of this paper, we consider first
the case when condition \eqref{sumConverges} is not necessarily fulfilled. 
\begin{proposition}
\label{Pr1}
Let $\Lambda=\{\lambda _n\}_{n \ge 0} \in (B)$ 
and let $\theta$ be an inner function. Suppose that $\mathcal K_{\Lambda, \theta}$ is complete in $K_\theta$.
Then $\sum_{n \ge 1} |\theta_n(\mu)|^2=\infty$ for every $ \mu \in \D \setminus \Lambda$.
\end{proposition}


If condition \eqref{sumConverges} is  fulfilled, we can say much more. 


In  \eqref{distNRKmuinfty2} below, we obtain that
\begin{multline*}
\dist_{H^2}^2(\widetilde k_\theta(\mu, \cdot), \spanmat  \mathcal K_{\Lambda, \theta})  \\
= \frac  { |B_\Lambda(\mu)| ^2} {1-|\theta(\mu)|^2} \lim_{n\to\infty} (1-|\theta_n(\mu)|^2)  \prod_{k=0}^{n-1} \frac {|1-\overline \gamma_k \theta_k(\mu)|^2}{1- | \gamma_k|^2},
\end{multline*}
where $\widetilde k_\theta(\mu, \cdot)=\frac {k_\theta(\mu, \cdot)}{\|k_\theta(\mu, \cdot)\|_2}$ is the normalized reproducing kernel in $K_\theta$.

Since $\spanmat \mathcal K_{\Lambda, \theta} \neq K_{\theta}$ if and only if for every $ \mu \in \D \setminus \{\Lambda\} $
we have $ k_{\theta}(\cdot,\mu) \notin \spanmat \mathcal K_{\Lambda, \theta}$ (see Lemma B below), we obtain that under condition \eqref{sumConverges}, 
$\mathcal  K_{\Lambda , \theta}$  is complete in $K_\theta$ if and only if  $ \lim_{n \to \infty} |\theta_n(\mu)| =1$ for some/every $ \mu \in \D \setminus \{\Lambda\} $.
 
\subsection{Functions with given Schur--Nevanlinna coefficients}

Given a sequence $\Lambda$ and  a sequence $\Gamma=\{\gamma_n\}_{n \ge 0} $  satisfying  \eqref{sumConverges} we are going to construct two functions: an inner function $\theta$ and a function  $h\in \mathcal B$, whose Schur--Nevanlinna coefficients with respect to
  $\Lambda$ coincide with $\Gamma$ and which satisfy some additional conditions on the corresponding Schur--Nevanlinna functions.


\begin{lemma}
\label{L2}
Let $\Lambda=\{\lambda_n\}_{n \ge 0}\in (B)$, 
$\Gamma=\{\gamma_n\}_{n \ge 0}$ be a sequence of  points in $\D$ satisfying condition \eqref{sumConverges}, $\mu \in \D \setminus \Lambda$. Then 
\newline\noindent {\rm (i)} there exists a function $h\in \mathcal B$ such that its sequence of 
Schur--Nevanlinna  coefficients with respect to $\Lambda$ coincides with  $\Gamma$ and 
\newline\noindent $\sup_{n \ge 0} \|h_n\|_\infty<1$,
where $h_n$ is the $ n$-th  Schur--Nevanlinna function  for $(\Lambda,h)$;
\newline\noindent {\rm (ii)} there exists an inner function $\theta$ such that its sequence of Schur--Nevanlinna coefficients with respect to $\Lambda$ coincides with  $\Gamma$ 
and for every $\mu \in \D \setminus \Lambda$ we have
$\sup_n | \theta_n(\mu)|=1$,
  where $\theta_n$ is the $ n$-th  Schur--Nevanlinna function  for $(\Lambda, \theta)$.
\end{lemma}

It should be mentioned that, by the Schwarz lemma, $\sup_n | \theta_n(\mu)|=1$ for some $\mu \in \D $ if and only if  $\sup_n | \theta_n(\mu)|=1$ for every $\mu \in \D$.

Using Lemma~\ref{L2} we obtain the following result on the distances between $P_\theta l_\mu $ and $ \spanmat\mathcal  K_{\Lambda , \theta}$.
 
\begin{proposition}
\label{Th3}
Let $\Lambda=\{\lambda_n\}_{n \ge 0}\in (B)$, $\Gamma=\{\gamma_n\}_{n \ge 0}$ be a sequence of  points in $\D$ satisfying condition \eqref{sumConverges}.
Then
\newline\noindent {\rm (i)} there exists an inner function $\theta$ such that its sequence of Schur--Nevanlinna coefficients with respect to $\Lambda$ coincides with  $\Gamma$ and for every $w \in \D $ we have
$$
\lim_{n \to \infty} |\theta_n(w)|<1,
$$
and
$$
\dist_{H^2}^2(P_\theta l_{w}, \spanmat\mathcal  K_{\Lambda , \theta})  \asymp 1-\lim_{n \to \infty} |\theta_n(w)|^2
$$
uniformly in $w \in \D\setminus \Lambda$;
\newline\noindent {\rm (ii)} there exists an inner function $\theta$ such that its sequence of Schur--Nevanlinna coefficients with respect to $\Lambda$ coincides with  $\Gamma$ and  for every $w \in \D$ we have  
$$\lim_{n \to \infty}  |\theta_n(w)|=1
$$ 
and 
$$
P_\theta l_w \in \spanmat\mathcal  K_{\Lambda , \theta}.
$$                                                                                                                                                 
\end{proposition}


We say that a sequence $\Lambda=\{\lambda_n\}_{n\ge 0} $ of points in $\D$ satisfies the Carleson condition and write 
$\Lambda \in (C)$ if  
$$ 
\inf_n \prod_{m \neq n}\Bigl| \frac {\lambda_m-\lambda_n}{1-\overline \lambda_m \lambda_n}\Bigr| > 0. 
$$

Set 
\begin{gather*}
\widetilde k(\lambda, z)= \frac {k(\lambda, z)}{\|k(\lambda, z)\|_2}= \frac {\sqrt{1-|\lambda|^2}} {1-\overline {\lambda} z},\\
\widetilde  {\mathcal K}_{\Lambda}=\{\widetilde k(\lambda, \cdot)\}_{\lambda \in \Lambda},\quad \widetilde  {\mathcal K}_{\Lambda, \theta}=\{\widetilde k_\theta(\lambda, \cdot)\}_{\lambda \in \Lambda}.
\end{gather*} 
Notice that the Carleson condition on $\Lambda$ is sufficient for the family $  \widetilde  {\mathcal K}_{\Lambda}$ 
to be a Riesz sequence and is necessary  for the families $  {\mathcal K}_{\Lambda}$, $   {\mathcal K}_{\Lambda, \theta}$ to be uniformly minimal
(see \cite[Vol.2, Part D, Lemma 4.4.2]{11} for $ {\mathcal K}_{\Lambda}$ and \cite{5} for  ${\mathcal K}_{\Lambda, \theta}$).

The following result shows that for a fixed  $\Lambda  \in (C) $ different geometric behavior of ${\mathcal K}_{\Lambda, \theta}$ is possible.
\begin{theorem}
\label{Th4}
Let $\Lambda  \in (C) $, $\Gamma=\{\gamma_n\}_{n \ge 0}$ be a sequence of  points in $\D$ satisfying condition \eqref{sumConverges}. Then
\newline\noindent {\rm (i)} there exists an inner function $\theta$ such that its sequence of Schur--Nevanlinna coefficients with respect to $\Lambda$ coincides with  $\Gamma$,  $  \widetilde  {\mathcal K}_{\Lambda, \theta}$ is a non complete in $K_\theta$ Riesz sequence, 
for every $\mu \in \D\setminus \Lambda$ we have
$$
\lim_n |\theta_n(\mu)|<1, 
$$
and
$$
 \frac {\dist_{H^2}^2(P_{\theta} \widetilde  {k}(\mu, \cdot), \spanmat\mathcal  K_{\Lambda , \theta})}{\dist_{H^2}^2( \widetilde k(\mu, \cdot), K_{B_\Lambda})} \asymp 1-\lim_n |\theta_n(\mu)|^2, 
$$
uniformly in $\mu \in \D\setminus \Lambda$;
\newline\noindent {\rm (ii)} there exists an inner function $\theta$ such that its sequence of Schur--Nevanlinna coefficients with respect to $\Lambda$ coincides with  $\Gamma$, $  \widetilde  {\mathcal K}_{\Lambda, \theta}$ is a Riesz basis in $K_\theta$,
 and for every  $\mu \in \D\setminus \Lambda$ we have $\lim_n |\theta_n(\mu)|=1$.    
\end{theorem}

It is interesting to relate Theorem~\ref{Th4} to
the following question posed by Nikolski and partially resolved by Dyakonov \cite{10} : given an arbitrary inner function $\theta$,  does the corresponding model space $K_\theta$ always possess an unconditional basis of reproducing kernels ?
In other words ({see Theorem~C} below), given an inner function $\theta$, does there always exist a set $\Lambda$ satisfying the Carleson condition such that 
$ \dist_{L^\infty}(\theta, B_\Lambda H^\infty)<1$ and $ \dist_{L^\infty}(B_\Lambda, \theta H^\infty)<1$ ?


\subsection{Vanishing condition}
Here we consider the situation when  
 \begin{equation}
 \lim_{n \to \infty}   \theta(\lambda_n)=0.
 \label{limzero}
 \end{equation}
 
 We show that  condition \eqref{limzero} characterizes compact Hankel operators with symbols of the form $\theta \overline {B_\Lambda}$, $\Lambda \in (C)$:
 
 \begin{proposition} 
 \label{L5}
Let $\Lambda=\{\lambda_n\}_{n \ge 0} \in (C)$, $\theta$ be an  inner function. Then the following 
assertions are equivalent.
\newline\noindent {\rm (i)} $\lim_{n\to\infty} \theta(\lambda_n)=0$. 
\newline\noindent {\rm (ii)}  The Hankel operatop $H_{\theta \overline {B_\Lambda}}$ is compact.
\newline\noindent {\rm (iii)} The operator $\theta(T_{B_\Lambda}) \vert K_{B_\Lambda}$ (see Section 2.4 below) is compact.
\end{proposition}

It is known (see \cite{2}) that if $\Lambda \in (C)$,  $\theta$ is not a pure Blaschke product, and condition \eqref{limzero} is fulfilled, then 
$\widetilde {\mathcal  K}_{\Lambda, \theta } $ is a Riesz  basis in the closure of its linear hull, and it is not complete in
$K_\theta$. If 
$\theta$ is a pure Blaschke product, then $\widetilde {\mathcal  K}_{\Lambda, \theta } $ can be a Riesz basis, for example, in a special situation described in the following theorem.

\begin{theorem}
\label{Th6}
Let $\Lambda=\{\lambda_n\}_{n \ge 0} \in (C)$, 
$\mathcal M=\{\mu_n\}_{n \ge 0}  \in (C)$,  
be such that 
$\lim_{n \to \infty}B_\mathcal M(\lambda_n)=0$. Suppose that 
$\widetilde {\mathcal  K}_{\Lambda, B_\mathcal M}$ is a Riesz basis in $K_{B_\mathcal M}$. Then
$\lim_{n \to \infty}B_\Lambda(\mu_n)=0$ and
$\widetilde {\mathcal  K}_{\mathcal M, B_\Lambda}$ is a Riesz basis in $K_{B_\Lambda}$.  
\end{theorem}

Combining Lemma~H below 
 with Theorem~\ref{Th6} and using the fact that if ${\mathcal  K}_{\Lambda, \theta}$ is uniformly minimal, then $\Lambda \in (C)$ (see \cite[Corollary 3.4]{5}), we deduce
%
%

\begin{corollary} 
\label{Cor7}
The following assertions are equivalent:
\newline\noindent {\rm (i)} $\Lambda \in (C)$, $\mathcal M \in (C)$, 
$\lim_{n \to \infty}B_\mathcal M(\lambda_n)=0$, $\lim_{n \to \infty}B_\Lambda(\mu_n)=0$,
$\widetilde {\mathcal  K}_{\Lambda, B_\mathcal M}$ is a Riesz basis in $ K_{B_\mathcal M}$, $\widetilde{\mathcal K}_{\mathcal M, B_\Lambda}$ is a Riesz basis in $K_{B_\Lambda}$. 
\newline\noindent {\rm (ii)} $\Lambda \in (C)$, $\mathcal M \in (C)$, 
$\lim_{n \to \infty}B_\mathcal M(\lambda_n)=0$,
$\widetilde {\mathcal K}_{\Lambda, B_\mathcal M}$ is a Riesz basis in $K_{B_\mathcal M}$.
\newline\noindent {\rm (iii)} $\Lambda \in (C)$, $\mathcal M \in (C)$, 
$\lim_{n \to \infty}B_\mathcal M(\lambda_n)= 0$, 
$\mathcal K_{\Lambda, B_\mathcal M}$ is complete and minimal in  $K_{B_\mathcal M}$.
\newline\noindent {\rm (iv)} $\Lambda \in (B)$, $\mathcal M \in (C)$, 
$\lim_{n \to \infty}B_\mathcal M(\lambda_n)= 0$, 
$\mathcal K_{\Lambda, B_\mathcal M}$ is complete and uniformly minimal in  $K_{B_\mathcal M}$.
\end{corollary}

By Proposition~\ref{Pr1}, if one of the assertions of the corollary is fulfilled, then $\sum_{n \ge 0} |\theta_n(\mu)|^2=\infty$, $\mu \in \D \setminus \Lambda$.

Sometimes it is convenient to work with projections of Malmquist--Walsh functions $\{l_n\}_{n \ge 0}$, $ l_n=B_{0,n-1}\widetilde k(\lambda_n, \cdot)$, $ n \ge 0$,
where $B_{0,n}=\prod_{i=0}^n b_{\lambda _i}$, $b_\lambda=(-\lambda/|\lambda|)\tau_\lambda$ is the Blaschke factor corresponding to $\lambda$,
  instead of projections of reproducing kernels. Proposition~\ref{L5} and Lemma~G 
   below show that if $\lim_{n\to\infty} \theta(\lambda_n)=0$ then the Gram matrix of the family $\{P_\te l_n\}_{n \ge 0}$
  is of the form $I + K$ with $K$ compact. 
  
  Families whose Gram matrices have such a form are asymptotically orthonormal families (see the definition in Section 2 below) and are widely studied 
  (see, for example, \cite{3,6,9,13,14,15}).
%
%
%
%
  

In a general situation it may be complicated to estimate the norms $\| P_\theta l_n\|_2$. However, if $\lim_{n\to\infty} \theta(\lambda_n) = 0$, we can do this.


\begin{corollary}
\label{Cor8}
Let $\Lambda \in (C)$. Then the following assertions are equivalent
\newline\noindent {\rm (i)} $\lim_{n\to\infty} \theta(\lambda_n)=0$. 
\newline\noindent {\rm (ii)}  The Gram matrix of the family $\{P_\te l_n\}_{n \ge 0}$
  is of the form $I + K$ with $K$ compact. 

If these assertions are fulfilled, we have 
$$
\lim_{n \to \infty }\dist_{H^2} ^2(P_\theta l_{n}, \spanmat \mathcal  K_{\Lambda_{n-1}, \theta})  
= 1
$$
and thus
$\lim_{n\to\infty} \| P_\theta l_n\|_2 =1$.
\end{corollary}

%
 
\section{Necessary facts and definitions}
\label{sect2}
\subsection{Carleson's Interpolation theorem}
A sequence $\{z_n\}_{n \ge 1} \subset \D$ is called an interpolating sequence if
every interpolation problem
$$
f(z_n) = a_n, \qquad  n \ge 1,
$$
with bounded data $\{a_n\}_{n \ge 1}$  has a solution $f \in H^\infty$. 

A famous theorem of L.Carleson (\cite{12}) states that  a sequence
is interpolating if and only if it satisfies the Carleson condition $(C)$.

\subsection{Families of reproducing kernels and their projections}
 A family of vectors $\{x_n\}_{n \ge 1}$ in a Hilbert space $H$ is minimal if each element of the family lies outside the
closed linear hull of the others, and is uniformly minimal if $\inf_{n \ge 1} \dist_H (\frac {x_n}{\|x_n\|_H}, \spanmat \{x_k\}_{ k \neq n})>0$ .

Following, for example,  \cite[Chapter VI, Section 1, p.132] {1}, we say that
a family of  vectors $X=\{x_n\}_{n\ge 1}$ in a Hilbert space $H$ is called a Riesz sequence 
 if there exists a bounded on $\spanmat X$ and invertible  linear operator $V$ mapping  $X$ into an orthonormal family $\{Vx_n\}_{n\ge 1}$. Such an operator $V$ is called an orthogonalizer of $X$.

A  Riesz sequence complete in $H$  is called a Riesz basis. 

%

%


A family of vectors $\{x_n\}_{n\ge 1}$ in $H$ is called an asymptotically orthonormal  sequence  (AOS)
if for every $ N \ge 1$ there exist positive numbers $c_N$, $C_N$  
such that for every  $\{a_n\}_{n=1}^\infty \in l^2$ we have
$$
\label{aob}
c_N \sum_{n \ge N}|a_n|^2 \le \| \sum_{n \ge N }a_nx_n\|_H^2 \le C_N \sum_{n \ge N}|a_n|^2,
$$
where 
$$
\label{aobc}
\lim_{N \to \infty} c_N=\lim_{N \to \infty} C_N=1.
$$
A complete asymptotically orthonormal sequence is called an asymptotically orthonormal basis.


A sequence $\Lambda=\{\lambda_n\} $ of points of the open unit disc is thin if 
$$
 \lim_{n\to\infty} \prod_{m \neq n}\Bigl| \frac {\lambda_m-\lambda_n}{1-\overline \lambda_m \lambda_n}\Bigr|=1. 
$$

\begin{theoremAOS}
Let $\Lambda=\{\lambda_k\}_{k \ge 1}$, $\Lambda_{n, \infty}=\{\lambda_n\}_{k \ge n}$.
\begin{itemize}
\item[(1)] The sequence $\widetilde  {\mathcal K}_{\Lambda}$ is an AOS if and only if $\Lambda$ is thin (\cite {8,14}).
\item[(2)] If the sequence $\widetilde  {\mathcal K}_{\Lambda, \theta}$ is an AOS then $\Lambda$ is thin (\cite[Proposition 5.1] {3}).
\item[(3)]  \cite[Theorem 5.2] {3} Let $\sup_{\lambda \in \Lambda} |\te(\lambda)|<1$ and let the sequence $\widetilde  {\mathcal K}_{\Lambda}$ be an AOS. Then either
\begin{itemize}
\item[a)] $\widetilde  {\mathcal K}_{\Lambda, \theta}$ is an AOS

or

\item[b)] for some $p \ge 2$ the family $\widetilde  {\mathcal K}_{\Lambda_{p, \infty}, \theta}$ is an asymptotically orthonormal basis in $K_\theta$.
\end{itemize}
\end{itemize}
\end{theoremAOS}

We say that an operator $V$ acting between two Hilbert spaces is a $U+\mathfrak{S}_\infty$ operator ($V \in (U+\mathfrak{S}_\infty)$) if $V=U+K$ with $U$ unitary, $K$ compact.

We say that a matrix $A$ is an $I+\mathfrak{S}_\infty$ matrix ($A \in (I+\mathfrak{S}_\infty)$), if it defines a bounded invertible operator on $l^2$ having
the form $I+K$ where $K$ is a compact and $I$ is the identity operator.

Following \cite{8}, we say that a  Riesz sequence  $X$ is a $U+\mathfrak{S}_\infty$ family ($X \sim U+\mathfrak{S}_\infty $),
if  one (and thus, every) orthogonalizer of $X$ is a $U+\mathfrak{S}_\infty$ operator. 

 \begin{theoremUK} (see \cite{15, 9})
Let  $X$
be a family of vectors $\{x_n\}_{n\ge 1}$ in a Hilbert space.
We denote by $G(X)$ the Gram matrix of $X$.

The following assertions are equivalent :
\begin{itemize}
\item[1.] $X$ is an AOS.
\item[2.] $X \sim U+\mathfrak{S}_\infty$.
\item[3.]$G(X)\in (I+\mathfrak{S}_\infty)$.
\end{itemize}
%
\end{theoremUK}

 %
%
%
%
%
%
%

Next we need several results which can be found, for example, in \cite{1,11}. In what follows $ \Lambda \subset \D $, $\Lambda \in (B)$, $B=B_\Lambda$, $\theta$ is an inner function.

\begin{theoremletterA}  (see \cite[Lecture VI, pp. 132--135]{1})
\label{ThA}
\begin{itemize}
\item [(1)] The following assertions are equivalent :
\begin{itemize}
\item [(a)] $\widetilde {\mathcal K}_\Lambda$ is a Riesz basis in $K_B$.
\item [(b)] ${\mathcal K}_\Lambda$ is a uniformly minimal family.
\item[(c)] $\Lambda \in (C)$. 
\end{itemize}
\item [(2)] Let $\mu \in \D \setminus \Lambda$. Then $\dist_{H^2}( \widetilde k(\mu, \cdot), K_B)=|B(\mu)|$.
\end{itemize}
\end{theoremletterA}

\begin{lemmaletterB} \cite[p. 211]{1}.
\label{LB}
Let  $\spanmat \mathcal K_{\Lambda, \theta} \neq K_{\theta}$. Then for every  $\mu \in \D \setminus \{\Lambda\} $
we have $ k_{\theta}(\mu, \cdot) \notin \spanmat \mathcal K_{\Lambda, \theta}$.

\end{lemmaletterB}

\begin{theoremletterC} \cite[Vol.2, Part D, Theorem 4.4.6]{11}.
\label{ThC}
Let $\sup_{\lambda \in \Lambda} |\theta(\lambda)|<1$. The following assertions are equivalent:
\begin{itemize}
\item [(1)] The family $\widetilde  {\mathcal K}_{\Lambda, \theta}$ is a Riesz basis in $K_\theta$.
\item [(2)] $\Lambda \in (C)$ and the projection $P_{\theta} \vert K_B$ is an isomorphism onto $K_{\theta}$ (in particular, $\spanmat \mathcal K_{\Lambda, \theta}= K_{\theta}$). 
\item [(3)]  $\Lambda \in (C)$ and $\dist_{L^\infty}(\theta, BH^\infty)<1$, $\dist_{L^\infty}(B, \theta H^\infty)<1$.
\item [(4)] $\Lambda \in (C)$, $\dist_{L^\infty}(\theta, B H^\infty)<1$, and for some/every $\mu \in \D \setminus \Lambda$  we have $\dist_{L^\infty}(\theta , b_\mu B H^\infty)=1$.
\end{itemize}
\end{theoremletterC}


\begin{theoremletterD}\cite[Vol.2, Part D, Theorem 4.4.8]{11}.
\label{ThD}
Let $\sup_{\lambda \in \Lambda} |\theta(\lambda)|<1$. The following assertions are equivalent:
\begin{itemize}
\item [(1)]  The family $\widetilde  {\mathcal K}_{\Lambda, \theta}$ is a Riesz sequence.
\item [(2)] $\Lambda \in (C)$ and the projection $P_{\theta} \vert K_B$  is an isomorphism  into $K_{\theta}$. 
\item [(3)] $\Lambda \in (C)$ and $\dist_{L^\infty}(\theta, BH^\infty)<1$.
\end{itemize}
\end{theoremletterD}

Combining two previous theorems, we get the following  result.

\begin{propositionletter?} 
\label{PE}
Let  $\sup_{\lambda \in \Lambda} |\theta(\lambda)|<1$. The following assertions are equivalent:
\begin{itemize}
\item [(1)]  The family $\mathcal K_{\Lambda, \theta}$ is a Riesz basis in the closure of its linear hull and is not complete in $K_\theta$.
\item [(2)]  $\Lambda \in (C)$ and the projection $P_{\theta} \vert K_B$  is an isomorphism onto a proper closed subspace of $K_{\theta}$. 
\item [(3)]  $\Lambda \in (C)$ and $\dist_{L^\infty}(\theta, BH^\infty)<1$, $\dist_{L^\infty}(B, \theta H^\infty)=1$.
\item [(4)]   $\Lambda \in (C)$ and for some/every $\mu \in \D$  we have \newline\noindent $\dist_{L^\infty}(\theta , b_\mu B H^\infty)<1$.
\end{itemize}
\end{propositionletter?}

\subsection{Distances}

%
%


Let  $\theta$ be an inner function, $\Lambda=\{\lambda _n\}_{n \ge 0} \in (B)$, $\mu \in \D \setminus \Lambda $, $ \Lambda_n=\{\lambda _k\}_{k=0}^n$, $B_{k,n}=\prod_{j=k}^n b_{\lambda_j}$.

Furthermore, let $L=\{l_n\}_{n \ge 0}$ be the Malmquist--Walsh basis of $K_{B_\Lambda}$, $l_0=\widetilde k(\lambda_0, \cdot)$, 
$ l_n=B_{0,n-1}\widetilde k(\lambda_n, \cdot)$, $ n \ge 1$.

Denote by $\{\gamma_n\}_{n \ge 0} $, $\{\theta_n\}_{n \ge 0} $  the Schur--Nevanlinna coefficients and functions corresponding to $\Lambda, \theta$.
Here are several distance formulas from \cite{5} we need in this paper.
\begin{itemize}
\item [(1)]  By \cite[Proposition 3.1]{5}, we have
\begin{multline}
\dist_{H^2} ^2(P_\theta l_{n}(z), \spanmat \mathcal  K_{\Lambda_{n-1}, \theta})  \\
= (1-|\theta(\lambda_n)|^2) \prod_{k=1}^n \frac {1-|\theta_k(\lambda_n)|^2}{1-|\theta_k(\lambda_n)|^2|b_{\lambda_{k-1}}(\lambda_n)|^2}.
\label{distMWlambdan}
\end{multline}

\item [(2)]  By \cite[Proposition 3.2]{5} 
we have
\begin{equation}
\dist_{H^2} ^2(\widetilde k_{\theta}(\mu, \cdot), \spanmat \mathcal  K_{\Lambda_n, \theta})  
= |B_{0,n}(\mu)| ^2\prod_{k=1}^n \frac {1-|\theta_k(\mu)|^2}{1-|\theta_k(\mu)|^2|b_{\lambda_{k-1}}(\mu)|^2}.
\label{distNRKmun}
\end{equation}
\item [(3)]Passing to the limit $n \to \infty$, we conclude that 
\begin{equation}
\dist_{H^2} ^2(\widetilde k_{\theta}(\mu, \cdot), \spanmat  \mathcal K_{\Lambda, \theta})  
= |B_\Lambda(\mu)| ^2\prod_{k=1}^\infty \frac {1-|\theta_k(\mu)|^2}{1-|\theta_k(\mu)|^2|b_{\lambda_{k-1}}(\mu)|^2}.
\label{distNRKmuinfty}
\end{equation}

and thus
\begin{multline}
\dist_{H^2} ^2(\widetilde k_{\theta}(\lambda_n, \cdot), \spanmat  \mathcal K_{\Lambda \setminus \{\lambda_n\}, \theta})\\ 
=|B_{\Lambda \setminus \{\lambda_n\}}(\lambda_n)| ^2\prod_{k \neq n}^\infty \frac {1-|\theta_k(\lambda_n)|^2}{1-|\theta_k(\lambda_n)|^2|b_{\lambda_{k-1}}(\lambda_n)|^2}.
\label{distNRKlambdaninfty}
\end{multline}

\end{itemize}


The  following useful identity can be verified by induction:
\begin{multline}
(1-|\theta(\mu)|^2)  \prod_{k=1}^{n} \frac {1-|\theta_k(\mu)|^2}{1-|\theta_k(\mu)|^2|b_{\lambda_{k-1}}(\mu)|^2}\\= (1-|\theta_n(\mu)|^2) \prod_{k=0}^{n-1} \frac {|1-\overline \gamma_k \theta_k(\mu)|^2}{1- | \gamma_k|^2}.
\label{usefulidentity}
\end{multline}
%

Combining \eqref{usefulidentity} with  \eqref{distNRKmun} 
we obtain
\begin{multline}
\dist_{H^2} ^2(\widetilde k_\theta(\mu, \cdot), \spanmat \mathcal  K_{\Lambda_n, \theta})  \\
= \frac  { |B_{0,n}(\mu)| ^2} {1-|\theta(\mu)|^2} (1-|\theta_n(\mu)|^2) \prod_{k=0}^{n-1} \frac {|1-\overline \gamma_k \theta_k(\mu)|^2}{1- | \gamma_k|^2}.
\label{distNRKmun2}
\end{multline}
Once again, passing to the limit $n \to \infty$, we conclude that
\begin{multline}
\dist_{H^2} ^2(\widetilde k_\theta(\mu, \cdot), \spanmat  \mathcal K_{\Lambda, \theta})  \\
= \frac  { |B_\Lambda(\mu)| ^2} {1-|\theta(\mu)|^2} \lim_{n\to\infty} (1-|\theta_n(\mu)|^2)  \prod_{k=0}^{n-1} \frac {|1-\overline \gamma_k \theta_k(\mu)|^2}{1- | \gamma_k|^2}.
\label{distNRKmuinfty2}
\end{multline}

\subsection{Hankel operators and model operators}
Let ${{\varphi}} \in L^\infty$. The Hankel operator  $H_\varphi : H^2 \to H^2_{-}$ is defined by
$$
H_{\varphi} f := P_-  M_{  \varphi} f, \qquad f \in H^2, 
$$
where  $H_{-}^2=L^2 \ominus  H^2$, $P_-$ is the projection onto $H_{-}^2$, $M_{a} $ is the operator of multiplication by the function $a$.

Here are several  facts we use later on. 
We denote by $C$ the space $C(\T)$ of continuous functions on $\T$.

\begin{thmHa} \cite[Vol.1, Part B, Chapter 2, p.214]{11}. 
The Hankel operator $H_{ \varphi}$ is compact if and only if ${\varphi} \in H^\infty+C$.
\end{thmHa}

\begin{lemmaletterF} \cite[Appendix 4, p.311]{1}. 
\label{LF}
Let ${\varphi}  \in H^\infty+C$, $|{{\varphi}}|=1$ a.e. $\T$, and $\dist_{ L^\infty}(\overline {\varphi}, H^\infty+C)<1$. Then 
$\overline  {\varphi}\in H^\infty+C$.
\end{lemmaletterF}


Let $V$ be an inner function. 
The model operator  $T_V:K_V\to K_V$ is defined by
$$
T_Vg:=P_{V} M_z g,  \qquad g \in K_V.
$$

We need the functional calculus for the model operator (see \cite[Lecture III]{1}): 
given $\psi \in H^\infty$, the mapping $\psi \to \psi (T_V)$ is defined by 
$\psi (T_V)=P_VM_\psi\vert K_V$, $[\psi (T_V)]^*=P_+M_{\overline {\psi }}\vert K_V$.  
The following formula relates the Hankel and the model operators (see \cite[Appendix 4, p.302, formula (1)]{1}):
\begin{equation}
\psi (T_V)P_V = M_VH_{\psi \overline {V}}, \qquad \psi \in H^\infty.
\label{connectionFormula}
\end{equation}

\begin{lemmaletterG} \cite[Lemma 2.2]{5}. 
\label{LG}
Let $\theta$ be an inner function, $\Lambda \in (B)$, and let $L=\{l_n\}_{n \ge 0}$ be the  Malmquist--Walsh basis  corresponding to $\Lambda$.
Then the Gram matrix of the family $ \{P_\theta l_n\}_{n \ge 0}$  is equal to $I-AA^*$, where $A$ is the matrix of $\theta(T_{B_\Lambda})$ in the basis $L$.
 \end{lemmaletterG}
 
\begin{lemmaletterH} \cite[Corollary 4.8]{5}. 
\label{LH}
Let $\Lambda \in (C)$  and let $\lim_{n\to\infty} \te(\lambda_n)=0$.
Then the following assertions are equivalent:
\begin{itemize}
\item[(1)]The family $\mathcal K_{\Lambda,\theta}$ is minimal.
\item[(2)]The family $\widetilde {\mathcal K}_{\Lambda,\theta}$ is a Riesz basis
in the closure of its linear hull.
\item[(3)] The class $\te+BH^\infty$ contains a non-extreme point of $\mathcal B$.
\item[(4)] $\dist_{L^\infty} (\te, BH^\infty)<1$.
\end{itemize}
\end{lemmaletterH}
 
\section{Functions with given Schur--Nevanlinna coefficients}

\renewcommand*{\proofname}{Proof of Proposition~\ref{Pr1}}
\begin{proof}
By Lemma~B 
and \eqref{distNRKmuinfty}, the completeness of $\mathcal K_{\Lambda, \theta}$ implies that for every $ \mu \in \D \setminus \Lambda $ we have 
$$
\prod_{n \ge 1} \frac {1-|\theta_n(\mu)|^2}{1-|\theta_n(\mu)|^2|b_{\lambda_{n-1}}(\mu)|^2}=0.
$$ 
Thus $\prod_{n \ge 1} {(1-|\theta_n(\mu)|^2)}=0$ and, hence, $\sum_{n \ge 1}  |\theta_n(\mu)|^2=\infty$.
\end{proof}

Consequently, if for some  $ \mu \in \D \setminus \Lambda $  we have $\sum_{n \ge 1}  |\theta_n(\mu)|^2 < \infty$, then $\mathcal K_{\Lambda, \theta}$ is not complete in $K_\theta$. Later on in this section we show that this sufficient condition is not necessary.


To proceed further on, we need to recall the following construction introduced  in \cite{4}.
\subsection {Inverse Schur--Nevanlinna process I}

Let 
$\Lambda=\{\lambda _n\}_{n \ge 0}  \in (B)$,  $\{\gamma_n\}_{n\ge 0}$ be  a sequence of points in $\D$.
For $n \ge 0$ we set
\begin{equation}
\left.
\begin{aligned}
h_{n,0}&=\gamma_n, \qquad n \ge 0, \\
h_{n,k}&=\tau_{-\gamma_{n-k}} (\tau_{\lambda_{n-k}} h_{n,k-1}), \qquad 1 \le k \le n.
\end{aligned}
\right \} 
\label{ISN1}
\end{equation}

One can easily verify (see also \cite[Lemma 2.1]{4}) that for $ 0\le k\le n$ we have
\begin{itemize}
\item[(1)] $\|h_{n,k}\|_\infty  < 1$.
\item[(2)] $h_{n, n-k} =(h_{n,n})_k$,  
where $(h_{n,n})_k$ is the $k$-th Schur--Nevanlinna function corresponding to $ \Lambda$, $ h_{n,n}$.
\end{itemize}


By induction, one verifies that
$$
1-|h_{n,n}(\zeta)|^2=(1-|\gamma_n|^2)\prod_{k=0}^{n-1}  \frac {1-|\gamma_k|^2}{|1+\overline {\gamma_k}\tau_{\lambda_k}(\zeta)h_{n,n-k-1} (\zeta)|^2}, \qquad \zeta \in \T .
$$

Under condition \eqref{sumConverges},
the latter equality gives
$$
1-|h_{n,n}(\zeta)|^2\ge (1-|\gamma_n|^2)\prod_{k=0}^{n-1}  \frac {1-|\gamma_k|^2}{(1+|\gamma_k|)^2}\ge c>0,\qquad \zeta \in \T,
$$
where the constant $c$ does not depend on $n$.

Let now $\theta$ be an inner function different from a finite Blaschke product, 
$\Lambda=\{\lambda _n\}_{n \ge 0}  \in (B)$. Let  $\{\gamma_n\}_{n\ge 0}$ be the sequence of the  Schur--Nevanlinna  coefficients corresponding to $\Lambda$, $\theta$.

Recall that $B_{k,n}=\prod_{j=k}^n b_{\lambda_j}$. For $ 0\le k\le n$ we have
\begin{itemize}
\item[(3)] $h_{n,n-k}\in \theta_k + B_{k,n} H^\infty$ (and, in particular, $h_{n,n} \in \theta+B_{0,n} H^\infty$).
\end{itemize}

By a normal family argument, there exists a subsequence $\{h_{n_k,n_k}\}$ of $\{h_{n,n}\}$ that
converges uniformly on compact subsets of the unit disc to a function $h \in  \theta+B_\Lambda H^\infty$. 

%
Then 
\begin{equation}
\|h\|_\infty<1
\label{star}
\end{equation} 
and, hence, $\dist_{H^\infty}(\theta,B_\Lambda H^\infty)=\dist_{L^\infty}(\theta \overline{B_\Lambda}, H^\infty)<1$. If, additionally, $\Lambda \in (C)$, Theorem~D 
shows that 
$\widetilde  {\mathcal K}_{\Lambda, \theta}$ is a Riesz sequence.


Next, we consider a modification of the previous construction.
\subsection {Inverse Schur--Nevanlinna process II}
Given $ \theta \in {\mathcal B}$ which is not a finite Blaschke product, $\Lambda=\{\lambda_n\}_{n\ge 0} \in (B)$, let $\{\gamma_n\}_{n\ge 0}$ be the Schur--Nevanlinna coefficients 
and let $\{\theta_n\}_{n\ge 0}$ be  the Schur--Nevanlinna functions 
corresponding to $\Lambda, \theta$.  Given $\mu \in \D \setminus \Lambda$, we define for $n \ge 0$ :
\begin{equation}
\left.
\begin{aligned}
h_{\mu;n,0} &=\theta_n(\mu), \\
h_{\mu;n,k} & =\tau_{-\gamma_{n-k}} (\tau_{\lambda_{n-k}} h_{\mu;n,k-1}), \qquad 1 \le k \le n.
\end{aligned}
\right \} 
\end{equation}

For $ 0\le k\le n$ we have 
\begin{itemize}
\item[(1)] $\|h_{\mu;n,k}\|_\infty  \le 1$.
\item[(2)] $h_{\mu;n, n-k} =(h_{\mu;n,n})_k$, 
where $(h_{\mu;n,n})_k$ is the $k$-th Schur--Nevan\-linna function corresponding to $\Lambda$, $ h_{\mu;n,n}$.
\item[(3)] $h_{\mu;n,n-k}\in \theta_k + b_\mu B_{k,n-1} H^\infty$
(and, in particular, $h_{\mu;n,n}\in \theta+b_\mu  B_{0,n-1} H^\infty$).

\end{itemize}

Again, by a normal family argument, 
there exists a subsequence $\{h_{\mu;n_k,n_k}\}$ of $\{h_{\mu;n,n}\}$ that
converges uniformly on compact subsets of the unit disc to a function $h_\mu \in  \theta+b_\mu B H^\infty$.  

Suppose that  condition \eqref{sumConverges} holds. By \eqref{distNRKmuinfty2}, $k_\theta(\mu, \cdot) \in \spanmat \mathcal K_{\Lambda, \theta}$ 
if and only if $\sup_n |\theta_n(\mu)| =1$. Furthermore, if   $\sup_n |\theta_n(\mu)| =1$ for some $\mu \in \D \setminus \Lambda$, then the same equality holds for every $\mu \in \D \setminus \Lambda$. In this case the family $\mathcal K_{\Lambda, \theta}$ is complete in $K_\theta$, 
and hence  for every $\mu \in \D \setminus \Lambda$ the family $\mathcal K_{\Lambda \cup \{\mu\}, \theta}$ is not minimal.   In \cite[Theorem 4.7]{5}, we proved that the latter is equivalent to the fact that $\|\theta+Bh\|_\infty>1$ if $h  \not\equiv 0$. 
Thus in this case $h_\mu$ is equal to the function $\theta$ itself. 

Now suppose that $\sup_n |\theta_n(\mu)| <1$ (and condition \eqref{sumConverges} still holds).  By construction of $h_{\mu;n,n}$ we have
\begin{multline*}
\label{1-hmucarre}
1-|h_{\mu;n,n}|^2=(1-|\theta_n(\mu)|^2)\prod_{k=0}^{n-1}  \frac {1-|\gamma_k|^2}{|1+\overline {\gamma_k}\tau_{\lambda_k}h_{\mu;n,n-k-1}|^2}
\\
\ge (1-|\theta_n(\mu)|^2)\prod_{k=0}^{n-1}  \frac {1-|\gamma_k|^2}{(1+|\gamma_k|^2)} \ge C(\mu)>0
\end{multline*}
on the unit circle, where the constant $C(\mu)$ does not depend on $n$ (but depends on $\mu$). 
Thus, $\|h_\mu\|_\infty<1$ and, hence, $\dist_{H^\infty}(\theta,b_\mu B_\Lambda H^\infty)<1$. 
If, additionally, $\Lambda \in (C)$, Proposition~E 
shows that $\mathcal K_{\Lambda, \theta}$ is a Riesz sequence which is not complete in $K_\theta$.

\subsection{Proof of Lemma~\ref{L2}}
To prove part (i) of Lemma~\ref{L2}, we start with a limit function $h$ of the inverse Schur--Nevanlinna process \eqref{ISN1}.

Let $\{h_n\}_{n\ge 0} $ be the Schur--Nevanlinna functions corresponding to $\Lambda, h$. The corresponding Schur--Nevanlinna  coefficients are $\{ \gamma_n \}_{n\ge 0} $.
By induction, we have
\begin{equation}
1-|h(\zeta)|^2= 
(1-|h_n(\zeta)|^2) \prod_{k=0}^{n-1} \frac {1-|\gamma_k|^2}{|1+\overline{\gamma_{k}}\tau_{\lambda_k}(\zeta)h_{k+1}(\zeta)|^2}, \quad \zeta \in \T.
\label{hinduction}
\end{equation}
By condition \eqref{sumConverges} we have \eqref{star}, and, again by condition \eqref{sumConverges}, we obtain 
 that  $\sup_n \|h_n\|_\infty<1$.

To prove part (ii) of Lemma~\ref{L2}, we use yet another modification of the inverse Schur--Nevanlinna process.

\begin{ISP3} 
Let 
$\Lambda=\{\lambda _n\}_{n \ge 0}  \in (B)$,  $\{\gamma_n\}_{n\ge 0}$ be  a sequence of points in $\D$.
For $n \ge 0$ we set
\begin{equation} 
\left  .
\begin{aligned}
h_{n,0}&=1,\\
h_{n,k}&=\tau_{-\gamma_{n-k}} (\tau_{\lambda_{n-k}} h_{n,k-1}), \qquad 1 \le k \le n.
\end{aligned}
\right \} 
\end{equation}
\end{ISP3}

%

For $ 0\le k\le n$ we have 
\begin{itemize}
\item[(1)] $\|h_{n,k}\|_\infty  = 1$.
\item[(2)] $h_{n, n-k} =(h_{n,n})_k$.
\item[(3)] $h_{ n, n-k}(\lambda_k)=\gamma_k$.
\end{itemize}

Thus the Schur--Nevanlinna coefficients of $h_{n,n}$ at the sequence of points $\{\lambda_0, \lambda_1, \dots , \lambda_{n-1}\}$ are $\gamma_0, \gamma_1, \dots , \gamma_{n-1}$.
Once again, there exists a subsequence $\{h_{n_k,n_k}\}$ of $\{h_{n,n}\}$ that converges uniformly on compact subsets of the unit disc to a function $h \in {\mathcal B}$.
From now on we write  $h_{n,n}$ instead of $h_{n_k,n_k}$. Denote by  $\{h_n$\}  the Schur--Nevanlinna functions corresponding to $\Lambda, h$.

Next we show that 
\begin{equation}
|h_n(z)| \to 1, \quad  z \in \D \setminus \Lambda.
\label{hnzto1}
\end{equation}

By definition, we have
\begin{multline*}
h(z)-h_{n,n}(z) = \frac {h_1\tau_{\lambda_0}(z) +\gamma_0}{1+\overline {\gamma_0} h_1 \tau_{\lambda_0} (z)} - \frac {h_{n,n-1}\tau_{\lambda_0}(z) +\gamma_0}{1+\overline {\gamma_0} h_{n,n-1} \tau_{\lambda_0}(z) }=\\
= \frac {(h_1(z)-h_{n,n-1}(z)) \tau_{\lambda_0}(z)(1-|\gamma_0|^2)} {(1+\overline {\gamma_0} h_1 \tau_{\lambda_0}(z))(1+\overline {\gamma_0} h_{n,n-1} \tau_{\lambda_0}(z))}, \quad z \in \overline \D, \quad n \ge 1.
\end{multline*}

By induction, we get (for  $z \in \overline \D$)
\begin{multline*}
h(z)-h_{n,n}(z) \\= (h_n(z) - 1)\prod_{k=0}^{n-1} \frac {(1-|\gamma_k|^2) \tau_{\lambda_k}(z)} {(1+\overline {\gamma_k} h_{k+1} \tau_{\lambda_k}(z))(1+\overline {\gamma_k} h_{n,n-k-1} \tau_{\lambda_k}(z))}.
\end{multline*}

Condition \eqref{sumConverges} gives
$$
|h(z)-h_{n,n}(z)| \asymp |h_n(z) - 1|\ |B_{0,n-1}(z)|, 
$$
 uniformly in $z \in \overline \D$ and $n \ge 1$.
Since  $\{h_{n,n}\}_{n \ge 0}$  converges to $h$ uniformly on compact subsets of $\D$, 
we have $\lim_{n \to \infty}|h_n(z)| = 1$  (even $\lim_{n \to \infty} h_n(z) = 1$),  $z \in \D \setminus \Lambda$.

On the other side, we have, as in \eqref{hinduction},
$$
1-|h(\zeta)|^2 = (1-|h_n(\zeta)|^2 ) \prod_{k=0}^{n-1} \frac {1-|\gamma_k|^2} {|1+\overline {\gamma_k}  \tau_{\lambda_k}(\zeta) h_{k+1}(\zeta)|^2}, \quad \zeta \in \T.
$$
Suppose that $|h(\zeta)|<1-\varepsilon$ on a set $e \subset \T$ of positive Lebesgue measure. Then by \eqref{sumConverges}, for every $n >0$ we have $|h_n(\zeta)|<1-\delta \varepsilon$ on $e$,
where the constant $\delta<1$ depends on the sequence $\{\gamma_n\}_{n \ge 0}$. Fix $z \in \D \setminus \Lambda$. Then the Jensen formula gives $|h_n(z)|<1-\delta \varepsilon \beta$ for some
$\beta=\beta(e,z)>0$ that contradicts to \eqref{hnzto1}.
Thus,  $h$ is inner and we can take $\theta=h$. \qed

\subsection{The proofs of Proposition~\ref{Th3} and Theorem~\ref{Th4}}

In what follows we need 
\begin{AAK} (see \cite[p.204]{1})
Let 
$f \in L^\infty$ and $\dist_{L^\infty}(f, H^\infty)<1$. 
Then $f+H^\infty$ contains a unimodular function.
\end{AAK}

\renewcommand*{\proofname}{Proof of Proposition~\ref{Th3}}
\begin{proof}
(i) Let $\Gamma=\{\gamma_n\}_{n \ge 0} $ satisfy condition \eqref{sumConverges}. Pick some $\mu \in \D \setminus \Lambda $  and let $h$ be the function constructed in Lemma~\ref{L2} (i). 
Since $\|h\|_{L^\infty} < 1$, we have  ${\|h \overline{b_\mu B_\Lambda}\|}_{L^\infty} < 1$, and the Adamyan--Arov--Krein theorem ensures the existence of  a unimodular function $\varphi \in h \overline{b_\mu B_\Lambda}+H^\infty$. Put $\theta :=\varphi b_\mu B_\Lambda$.
Then $\theta $ is inner and $\theta - h \in b_\mu B_\Lambda H^\infty$. Thus  $ {\theta}_n(\mu) = h_n(\mu)$, $n \ge 0$. Since $\sup_{n \ge 0}|h_n(\mu)|<1$,  we obtain that   $\sup_{n \ge 0}|h_n(w)|<1$ for every $w \in \D$.
Next, formula \eqref{distlmu} completes the proof.
\newline\noindent (ii) Let $\Gamma=\{\gamma_n\}_{n \ge 0}$ satisfy condition \eqref{sumConverges}. It suffices to take the function $\theta$  from Lemma~\ref{L2} (ii) and use formula  \eqref{distlmu}.
\end{proof}

Now we pass to 
projections of reproducing kernels.
As we have already mentioned, the Carleson condition is necessary for $ {\mathcal K}_{\Lambda, \theta}$ to be a Riesz sequence (and, in particular, the points $\lambda_i$ have to be distinct).

\renewcommand*{\proofname}{Proof of Theorem~\ref{Th4}}
\begin{proof}
(i) Pick some $\mu_0 \in \D \setminus \Lambda$ and let $\theta$ be the function  constructed in the proof of Proposition~\ref{Th3} (i).  Since  \eqref{sumConverges} is fulfilled, we have $\dist_{H^\infty}(\theta,BH^\infty)<1$.
Proposition~\ref{Th3} states that  $\sup_n |\theta_n(\mu_0)|<1$. By \eqref{distNRKmuinfty2}  and \eqref{sumConverges}, $k_\theta(\mu_0, \cdot) \notin \spanmat\mathcal  K_{\Lambda , \theta}$ and, thus, $  \widetilde  {\mathcal K}_{\Lambda, \theta}$ is not complete in $K_\theta$.
Since $\Lambda  \in (C) $ and $\dist_{H^\infty}(\theta,BH^\infty)<1$, $  \widetilde  {\mathcal K}_{\Lambda, \theta}$ is a Riesz sequence which is not complete in $K_\theta$. 

As we have already mentioned, the Schwarz lemma ensures that $\sup_n |\theta_n(\mu)|<1$ for every $\mu \in \D\setminus \Lambda$.
Since 
$$
\dist_{H^2}^2(P_{\theta} k(\mu, \cdot), \spanmat\mathcal  K_{\Lambda , \theta}) =   (1-|\theta(\mu)|^2)\dist_{H^2}^2(\widetilde k_\theta(\mu, \cdot), \spanmat  \mathcal K_{\Lambda, \theta}), 
$$ 
by \eqref{distNRKmuinfty2} and Theorem~A  
we obtain that
$$
\frac {\dist_{H^2}^2(P_{\theta} \widetilde  {k}(\mu, \cdot), \spanmat\mathcal  K_{\Lambda , \theta})}{\dist_{H^2}^2( \widetilde k(\mu, \cdot), K_B)}  =   \lim_{n\to\infty} (1-|\theta_n(\mu)|^2) \prod_{k=0}^{n-1} \frac {|1-\overline \gamma_k \theta_k(\mu)|^2}{1- | \gamma_k|^2}.
$$
Taking into account
\eqref{sumConverges}, we get for every $\mu \in \D \setminus \Lambda$ that 
$$
c(1-\lim_{n \to \infty} |\theta_n(\mu)|^2) \le \frac {\dist_{H^2}^2(P_{\theta} \widetilde  {k}(\mu, \cdot), \spanmat\mathcal  K_{\Lambda , \theta})}{\dist_{H^2}^2( \widetilde k(\mu, \cdot), K_B)}   \le C(1-\lim_{n \to \infty} |\theta_n(\mu)|^2).
$$
\newline\noindent (ii) Pick some $\mu_0 \in \D \setminus \Lambda$ and let $\theta$ be the function  constructed in the proof of Proposition~\ref{Th3} (ii). Then  use \eqref{distNRKmuinfty2} 
as in the proof of part (i).  
\end{proof}

\section{Compact Hankel operators with symbols $\theta \overline B$}

It is known (see \cite{2}) that if $\Lambda \in (C)$, $\lim_{n \to \infty}
\theta(\lambda_n)=0$, and $\theta$ is not a pure Blaschke product, then 
$\widetilde {\mathcal  K}_{\Lambda, \theta } $ is a Riesz  basis in the closure of its linear hull, and this hull does not coincide with $K_\theta$. It is easily seen that for a pure Blaschke product $\theta$ 
this result is not true: we can take two interpolating Blaschke 
products $B_\Lambda$ and $B_\mathcal M$ such that 
$\lim_{\mu \in \mathcal M, |\mu| \to 1}B_\Lambda(\mu)=0$ 
and $\lim_{\lambda \in \Lambda, |\lambda| \to 1}B_\mathcal M(\lambda)=0$ and remove, if necessary, a finite number of $\lambda \in \Lambda$ or $\mu \in \mathcal M$. Then by Theorem~C 
 and Carleson's interpolating theorem 
both families $\widetilde {K}_{\Lambda, B_\mathcal M}$  and $\widetilde {K}_{\mathcal M, B_\Lambda}$
are Riesz bases (in $K_{B_\mathcal M}$ and $K_{B_\Lambda}$ correspondingly). 
Theorem~\ref{Th6} shows that this example is the ``only possible" one.

\begin{proof}[Proof of Proposition~\ref{L5}]
(i) $\Rightarrow$ (ii) Let $\lim_{n\to\infty} \theta(\lambda_n)=0$. Since $H^\infty+C$ is an algebra,  $\overline {B_{0,n}}H^\infty \subset H^\infty+C$. 
We have
\begin{multline*}
\dist_{L^\infty}(\theta \overline {B_\Lambda}, H^\infty+C)
\le \lim_{n\to\infty} (\dist_{L^\infty}(\theta \overline {B_\Lambda}, \overline {B_{0,n}}H^\infty))\\
=\lim_{n\to\infty} (\dist_{L^\infty}(\theta,  {B_{n+1,\infty}}H^\infty)).
\end{multline*}
Since $\theta(\lambda_n) \to 0$, the latter limit is equal to zero by Carleson's Interpolation Theorem. 

Since $H^\infty+C$
 is closed in $L^\infty$ (see \cite[Vol.1, Part B, Corollary 2.3.1]{11}), we have  
$\theta \overline {B_\Lambda} \in H^\infty+C$. By Hartman's theorem 
$H_{\theta\overline {B_\Lambda}}$ is compact.

(ii) $\Rightarrow$ (iii)  Let $H_{\theta\overline {B_\Lambda}}$ be compact.  By \eqref{connectionFormula},  $\theta (T_ {B_\Lambda})P_ {B_\Lambda}$ is also compact
and thus $\theta(T_ {B_\Lambda})\vert K_ {B_\Lambda}$ is compact. 

(iii) $\Rightarrow$ (i) Let $\theta(T_ {B_\Lambda})\vert K_ {B_\Lambda}$ be compact. Then $A=[\theta (T_{B_\Lambda})]^*=P_+M_{\overline{\te}}\vert K_{B_\Lambda}$ is compact.
By Theorem~A 
the family  $\{\widetilde{k_{\lambda_n}}\}_{n\ge 0}$ is a Riesz basis and thus it tends weakly to zero. 
Since $A$ is compact, we have  $\|A\widetilde{k_{\lambda_n}\|}=|\theta(\lambda_n)| \to 0$.
\end{proof}



\begin{proof}[Proof of Theorem~\ref{Th6}]
By Theorem~C  
 we have
\begin{equation}
\dist_{L^\infty}(B_\mathcal M\overline {B_\Lambda}, H^\infty)<1
\label{dbll1}
\end{equation}
and
\begin{equation}
\dist_{L^\infty}(B_\Lambda \overline {B_\mathcal M}, H^\infty)<1.
\label{dbll1x}
\end{equation}
By Proposition~\ref{L5} we have $B_\mathcal M \overline {B_\Lambda} \in H^\infty+C$.
Since 
$$\dist_{L^\infty}(B_\Lambda \overline {B_\mathcal M}, H^\infty+C) \le
\dist_{L^\infty}(B_\Lambda \overline {B_\mathcal M}, H^\infty),$$ \eqref{dbll1x} gives
$\dist_{L^\infty}(B_\Lambda \overline {B_\mathcal M}, H^\infty+C)<1$, and we can use Lemma~F with $u=B_\mathcal M \overline {B_\Lambda}$. Thus, $B_\Lambda \overline {B_\mathcal M} 
\in H^\infty+C $. By Proposition~\ref{L5}, we get 
$\lim_{n \to \infty}B_\Lambda(\mu_n)=0$. Thus, $\sup_n |B_\Lambda(\mu_n)|<1$. By the hypothesis, we have   $\sup_n |B_\mathcal M(\lambda_n)|<1$, and 
by Theorem~C, 
the family
$\widetilde {\mathcal  K}_{\mathcal M, B_\Lambda}$ is a Riesz basis in $K_{B_\Lambda}$. 
\end{proof}

To prove Corollary \ref{Cor8} we need the following well-known result.

\begin{lemma}
\label{le9}
Let $X=\{x_n\}_{n\ge 1} \sim U+\mathfrak{S}_\infty$ 
and let $V$ be an orthogonalizer of $X$. Then 
\begin{itemize}
\item[(i)]$V^{-1} \in (U+\mathfrak{S}_\infty)$.
\item[(ii)] Let $X^*=\{x_n^*\}_{n\ge 1}$ be the family in $\spanmat X$  biorthogonal to $X$. Then $X^* \sim U+\mathfrak{S}_\infty$.
\item[(iii)] $\lim_{n \to \infty} \| x_n\|=1$.
\item[(iv)] $\lim_{n \to \infty }\dist  (x_n, \spanmat_{ k \neq n} \{x_k\}) = 1$.
\end{itemize}
\end{lemma}

\begin{proof}[Proof]
\begin{itemize}
\item[(i)]See, for example, \cite[Lemma 2.2(b)]{9}
%
\item[(ii)] Since $(V^{-1})^*$ is an orthogonalizer of $X^*$ (\cite{1}, Lection 6), we have $X^* \sim U+\mathfrak{S}_\infty$.

\item[(iii)]  
The family $\{e_n=Vx_n\}_{n \ge 0}$ is an orthonormal basis.
By $(i)$, we have $V^{-1}=U+K$, where $U$ is a unitary operator and $K$ is a compact one. Then
$$
\|x_n\|^2=(V^{-1}e_n,V^{-1}e_n)=(U e_n+K e_n, U e_n+K e_n).
$$
Since $\{e_n\}_{n \ge 0}$ tends weakly to zero, we have  $\lim_{n \to \infty} \|x_n\|^2 =1$.

\item[(iv)] Since (see, for example, \cite[page 132]{1})
$$
\dist ^2(x_n, \spanmat \{x_k\}_{ k \neq n}) = \| x_n^*\|^{- 1},
$$
(i) and (iii) give  $\lim_{n \to \infty}  \| x_n^*\|=1$. Therefore, \newline\noindent $\lim_{n \to \infty } \dist  (x_n, \spanmat_{ k \neq n} \{x_k\}) = 1$.
\end{itemize}
\end{proof}

%
%
\begin{proof}[Proof of Corollary~\ref{Cor8}]
$(i) \Rightarrow(ii)$  By Lemma~G,  
the Gram matrix of $\{P_\te l_n\}_{n \ge 0}$ is the matrix of the operator $I-\theta(T_{B_\Lambda}) \theta(T_{B_\Lambda})^*$  in the basis $ \{l_n\}_{n \ge 0}$.
By Proposition~\ref{L5}, the operator $\theta(T_{B_\Lambda})$ is compact and, hence, $\theta(T_{B_\Lambda}) \theta(T_{B_\Lambda})^*$ is also compact.

To prove $(ii) \Rightarrow(i)$ it is sufficient to use the fact that for a bounded operator $A$ the compactness of $AA^*$ is equivalent to the  compactness of $A$
(and apply again Lemma~G and Proposition~\ref{L5}). By Lemma~\ref{le9}, 
$\lim_{n \to \infty }\dist_{H^2} (P_\theta l_{n}, \spanmat \mathcal  K_{\Lambda_{n-1}, \theta}) = 1$. 

%

 \end{proof}

\begin{rem}
Let $\Lambda \in (C)$.
The condition $\lim_{n\to\infty} \theta(\lambda_n)=0$ does not imply the minimality of $\{P_\te l_n\}_{n \ge 0}$. However, if $\{P_\te l_n\}_{n \ge 0}$ is minimal,
then the (bounded linear) operator $P_\te l_n \rightarrow l_n $ is an orthogonalizer of the family $\{P_\te l_n\}_{n \ge 0}$ (it maps $\spanmat_{n \ge 0} \{P_\te l_n\}$ onto $K_{B_\Lambda}$).
\end{rem}


Indeed, 
by Theorem~D  and Lemma~H, 
the operator $P_{\theta} \vert K_B$  is an isomorphism  onto its image. Thus it is invertible on the image and the inverse operator $V$ is also an isomorphism.
By Lemma~G and
Proposition~\ref{L5}
we have $P_\theta=I+K$, where $K$ is a compact operator. By Lemma~9, $V=I+K_0$,
where $K_0$ is a compact operator.

\begin{rem}
Each of the conditions  
$$
\lim_{n \to \infty }\dist_{H^2} (P_\theta l_{n}, \spanmat \mathcal  K_{\Lambda_{n-1}, \theta})  
= 1
$$
(necessary for the family  $\{P_\theta l_{n}\}_{n \ge 0}$ to be an asymptotically orthonormal sequence) 
and   
$$
\lim_{n \to \infty } \dist_{H^2}  (\widetilde k_{\theta}(\lambda_n, \cdot), \spanmat  \mathcal K_{\Lambda \setminus \{\lambda_n\}, \theta}) = 1
$$
(necessary for the family $\widetilde {\mathcal  K}_{\Lambda, \theta} $ to be an asymptotically orthonormal sequence) 
implies that
$$
\prod_{k=1}^\infty \frac {1-|\theta_k(\lambda_n)|^2}{1-|\theta_k(\lambda_n)|^2|b_{\lambda_{k-1}}(\lambda_n)|^2}=1,
$$
see \eqref{distMWlambdan} and \eqref{distNRKlambdaninfty}.

Certainly, the families $\{P_\theta l_{n}\}_{n \ge 0}$ and $\widetilde {\mathcal  K}_{\Lambda, \theta} $ are not necessarily asymptotically orthonormal sequences simultaneously. For example, the thinness of
$\Lambda$ is necessary for the family $\widetilde {\mathcal  K}_{\Lambda, \theta} $ to be an asymptotically orthonormal sequence and is not necessary for the family $\{P_\theta l_{n}\}_{n \ge 0}$ to be an asymptotically orthonormal sequence.
Conversely, the condition $\lim_{n \to \infty} \theta(\lambda_n) =0$ is necessary for the family $\{P_\theta l_{n}\}_{n \ge 0}$ to be an asymptotically orthonormal sequence and is not necessary for the family $\widetilde {\mathcal  K}_{\Lambda, \theta} $ to be an asymptotically orthonormal sequence.
\end{rem}


\end{document}